\documentclass[12pt]{amsart}
\usepackage{amsmath,amsthm,amssymb,amsfonts}
\usepackage[nobysame,abbrev,alphabetic]{amsrefs}
\usepackage{color}

\newtheorem{thm}{Theorem}[section]

\newtheorem{lemma}[thm]{Lemma}

\theoremstyle{definition}

\newtheorem{theorem}[thm]{Theorem}
\newtheorem{proposition}[thm]{Proposition}
\def\mR {\mathbb {R}^n}

\def\dstyle{\displaystyle}

\theoremstyle{definition}

\newtheorem{rem}[thm]{Remark}

\numberwithin{equation}{section}

\begin{document}





\title[On the Sobolev-Poincar\'e inequality of CR-manifolds]{On the Sobolev-Poincar\'e inequality of CR-manifolds}

\author{Yi Wang}
\address{Department of Mathematics, Johns Hopkins University, Baltimore MD 21218}
\email{ywang@math.jhu.edu}
\thanks{The research of the author is partially supported
by NSF grant DMS-1612015.}
\author{Paul Yang}
\address{Department of Mathematics, Princeton University, Princeton, NJ 08540}
\email{yang@math.princeton.edu}
\thanks{The research of the author is partially supported
by NSF grant DMS-1509505.}

\date{}



\begin{abstract}
The purpose is to study the CR-manifold with a contact structure conformal to the Heisenberg group. 
In our previous work \cite{WY}, we have proved that if the $Q'$-curvature is nonnegative, and the integral of $Q'$-curvature is below the dimensional bound $c_1'$, then we have the isoperimetric inequality. In this paper, we manage to drop the condition on the nonnegativity of the $Q'$-curvature. 
We prove that the volume form $e^{4u}$ is a strong $A_\infty$ weight. As a corollary, we prove the Sobolev-Poincar\'e inequality on a class of CR-manifolds with integrable $Q'$-curvature.


\end{abstract}

\maketitle

\section{Introduction} \label{sect:intro}
On a four dimensional manifold, the Paneitz operator $P_4$ and the Branson's $Q$-curvature \cite{Branson} 
have many analogous properties as the Laplacian operator $\Delta_g$ and the Gaussian curvature $K_g$ on surfaces. 
The Paneitz operator is defined as $$P_g=\Delta^2+\delta(\frac{2}{3}Rg-2 Ric)d,$$
where $\delta$ is the divergence, $d$ is the differential, $R$ is the scalar curvature of $g$, and $Ric$
is the Ricci curvature tensor. 
The $Q$-curvature is defined as 
$$Q_g=\frac{1}{12}\left\{-\Delta R +\frac{1}{4}R^2 -3|E|^2 ,\right\}
$$
where $E$ is the traceless part of $Ric$, and $|\cdot|$ is taken with respect to the metric $g$.
The most important two properties for the pair $(P_g, Q_g)$ are that under the conformal change $g_{w}=e^{2w}g_0$, \\
1. $P_{g}$ transforms by $P_{g_w}(\cdot)=e^{-4w}P_{g_0}(\cdot)$;\\
2. $Q_{g}$ satisfies the fourth order equation
$$P_{g_{0}}w+2Q_{g_0}=2Q_{g_{w}}e^{4w}.$$
As proved by Beckner \cite{Beckner} and Chang-Yang \cite{CY}, the pair $(P_g, Q_g)$ also apperas in the Moser-Trudinger inequality for higher order operators. 

On CR-manifold, it is a fundamental problem to study the existence and analogous properties of CR invariant operator $P$ and curvature scalar invariant $Q$. Graham and Lee \cite{GL} has studied a fourth-order CR covariant operator with leading term $\Delta_b^2 +T^2$ and Hirachi \cite{Hirachi} has identified the $Q$-curvature which is related to $P$ through a change of contact form. However, although the integral of the $Q$-curvature on a compact three-dimensional CR-manifold is a CR invariant, it is always equal to zero. And in many interesting cases when the CR three manifold is the boundary of a strictly pseudoconvex domains, the $Q$-curvature vanishes everywhere. As a consequence, it is desirable to search for some other invariant operators and curvature invariants on a CR-manifold that are more sensitive in the CR geometry. 
The work of Branson, Fontana and Morpurgo \cite{BFM} aims to find such a pair $(P', Q')$ on a CR sphere. Later, the definition of $Q'$-curvature is generalized to all pseudo-Einstein CR-manifolds by the work of Case-Yang \cite{CY} and that of Hirachi \cite{Hirachi2}.  
The construction uses the strategy of analytic continuation in dimension by Branson \cite{Branson}, restricted to the subspace of the CR pluriharmonic functions. 
$$P'_4:=\lim_{n\rightarrow 1} \frac{2}{n-1} P_{4,n}|_{\mathcal{P}}.$$
Here $P_{4,n}$ is the fourth-order covariant operator that exists for every contact form $\theta$
by the work of Gover and Graham \cite{GG}. By \cite{GL}, the space of CR pluriharmonic functions $\mathcal{P}$ is always contained in the kernel of $P_4$. 

In this paper, we want explore the geometric meaning of this newly introduced conformal invariant $Q'$-curvature. 

In Riemannian geometry, a classical isoperimetric inequality on a complete simply connected surface $M^2$, called Fiala-Huber's \cite{Fiala}, \cite{Huber} isoperimetric inequality
\begin{equation}\label{FialaHuber}
Vol(\Omega)\leq \frac{1}{2(2\pi-\int_{M^2}K_g^+ dv_g)} Area(\partial \Omega)^2,
\end{equation}
where $K_g^+$ is the positive part of the Gaussian curvature $K_g$. Also $\int_{M^2}K_g^+ dv_g< 2\pi$ is the sharp bound for the isoperimetric inequality to hold. 

In \cite{YW15}, we generalize the Fiala-Huber's isoperimetric inequality to all even dimensions, replacing the role of the Gaussian curvature in dimension two by that of the $Q$-curvature in higher dimensions:

Let $(M^n,g)=(\mathbb{R}^n, g= e^{2u}|dx|^2)$ be a complete noncompact even dimensional manifold.
Let $Q^+$ and $Q^-$ denote the
positive and negative part of $Q_g$ respectively; and $dv_g$ denote the volume form of $M$. Suppose $g= e^{2u}|dx|^2$ is a ``normal" metric, i.e.
\begin{equation}\label{normal}u(x)=
\displaystyle \frac{1}{c_n}\int_{\mathbb{R}^n} \log \frac{|y|}{|x-y|} Q_{g}(y) dv_g(y) + C;
\end{equation}
for some constant $C$.
If
\begin{equation}\label{assumption1}
\alpha:= \int_{M^n}Q^{+}dv_g < c_n
\end{equation}
where $c_n=2^{n-2}(\frac{n-2}{2})!\pi^{\frac{n}{2}}$,
and \begin{equation}\label{assumption2}
\beta:=\int_{M^n}Q^{-}dv_g < \infty,
\end{equation}
then $(M^n,g)$ satisfies the isoperimetric inequality with isoperimetric constant depending only on $n, \alpha$ and
$\beta$.
Namely, for any bounded domain $\Omega\subset M^n$ with smooth boundary,
\begin{equation}\label{1.89}
|\Omega|_g^{\frac{n-1}{n}}\leq C(n, \alpha,\beta) |\partial \Omega |_g.
\end{equation}

In our previous paper \cite{WY}, we have studied the $Q'$-curvature and $P'$ operator, and proved that if $(\mathbb{H}^1, e^{u}\theta)$ for pluriharmonic function $u$ is a complete CR-manifold with nonnegative $Q'$ curvature and nonnegative Webster scalar curvature at infinity, if in addition $Q'$ curvature
satisfies \begin{equation}
\displaystyle \int_{\mathbb{H}^1} Q' e^{4u}  \theta\wedge d\theta< c'_1,
\end{equation}
then $e^{4u}$ is an $A_1$ weight. Here $c'_1$ is the constant in the fundamental solution of $P'$ operator (See \cite{WY}). As a corollary, we have derived the isoperimetric inequality on CR-manifold $(\mathbb{H}^1, e^{u}\theta)$:
\begin{equation} Vol(\Omega)\leq C Area(\partial \Omega)^{4/3}. \end{equation}
Here the constant $C$ is controlled by  
$c'_1-\int_{\mathbb{H}^1} Q' e^{4u}  \theta\wedge d\theta$. 
To prove this result, we notice that the class of pluriharmonic functions $\mathcal{P}$ is the relevant subspace of functions for the conformal factor $u$.

The purpose of the current paper is two-fold. We will first study the case when $Q'$ curvature is negative. Then we will discuss the general case when 
$Q'$ curvature does not have a sign. The main results of the paper are stated in the following.

\begin{thm}\label{main2}Let $(\mathbb{H}^1, e^{u}\theta)$ be a complete CR-manifold, where $\theta$ denotes the contact form on the 
Heisenberg group $\mathbb{H}^1$ and $u$ is a pluriharmonic funcion on $\mathbb{H}^1$. If the $Q'$-curvature is negative, and the Webster scalar curvature is nonnegative at infinity. 
If 
\begin{equation}
\displaystyle \int_{\mathbb{H}^1} Q' e^{4u}  \theta\wedge d\theta< \infty,
\end{equation}
then $e^{4u}$ is a strong $A_\infty$ weight.
\end{thm}


Note that $e^{4u}$ is the volume form of this conformal metric, where $4$ is the homogeneous dimension of $\mathbb H^1$. The descriptions of $A_1$ weight and strong $A_\infty$ weight will be in Section \ref{sect:preliminary}.

We will then discuss the case when the $Q'$-curvature does not have a sign.

\begin{thm}\label{main3}Let $(\mathbb{H}^1, e^{u}\theta)$ be a complete CR-manifold, where $\theta$ denotes the contact form on the 
Heisenberg group $\mathbb{H}^1$ and $u$ is a pluriharmonic funcion on $\mathbb{H}^1$. If the Webster scalar curvature is nonnegative at infinity,
If 
\begin{equation}\alpha:=\displaystyle \int_{\mathbb{H}^1} Q'^+ e^{4u}  \theta\wedge d\theta< c'_1,\end{equation}
and 
\begin{equation}\beta:=\displaystyle \int_{\mathbb{H}^1} Q'^- e^{4u}  \theta\wedge d\theta< \infty,\end{equation}
then $e^{4u}$ is a strong $A_\infty$ weight.
\end{thm}

As a corollary of Franchi-Lu-Wheeden \cite{FLW},  we will show that $(\mathbb{H}^1, e^{u}\theta)$ satisfies Sobolev-Poincar\'e inequality. We remark that on a CR-manifold 
$(\mathbb{H}^1, e^{u}\theta)$, the David-Semme's \cite{DS} type of isoperimetric inequality is still an open question for strong $A_\infty$ weights.

\begin{thm}\label{main4}Let $(\mathbb{H}^1, e^{u}\theta)$ satisfy the same assumptions as in Theorem \ref{main3}. Let $K$ be a compact subset of $\Omega$. Then there exists $r_0$ depending on $K, \Omega$, and $\{X_j\}$ such that if $B=B(x,r)$ is a ball with $x\in K$ and $0<r<r_0$, and if $e^{4u}$ is $A_p $ weight for some $1 \leq p<4$.  Let $\mu(x):= e^{4u} dx$, $\nu(x):=  e^{(4-p)u} dx$.
Then
 \begin{equation}
 (\frac{1}{\mu(B)} \int_B |f(x)- f_B |^q d\mu	)^{1/q} \leq cr (\frac{1}{\nu(B)} \int_B |\nabla_b f(x)|^p d\nu)^{1/p},
 \end{equation}
for any $f\in Lip (\bar{B})$, with $f_B= \frac{1}{\mu (B)}\int_B f(x) d\mu$. The constant $c$ depends only on $K, \Omega$, $\alpha, \beta, p$.


\end{thm}

\section{Preliminaries}\label{sect:preliminary}

On a Heisenberg group $\mathbb H^n$, one can also define the $A_p$ weight, in the same way as on the Euclidean space $\mR$.  For a nonnegative local integrable function $\omega$, we call it an $A_p$ weight $p>1$, if for all balls $B$ in $\mathbb H^n$
\begin{equation}\dstyle
\frac{1}{|B|} \int_B \omega (x)dx \cdot \left( \frac{1}{|B|}\large \int_B \omega(x)^{-\frac{p'}{p}} dx \right)^{\frac{p}{p'}}\leq C<\infty.
\end{equation}
Here $\frac{1}{p'}+ \frac{1}{p}=1$. The constant $C$ is uniform for all $B$. The definition of $A_1$ weight is given by taking the limit process $p\rightarrow 1$.
Namely, $\omega$ is called an $A_1$ weight, if 
\begin{equation}
\dstyle M \omega(x) \leq C \omega(x),
\end{equation}
for almost all $x\in B$.

An important property of $A_p$ weight is the reverse H\"older inequality:
if $\omega$ is an $A_p$ weight for some $p \geq 1$, then there exist an $r>1$ and a $C>0$ such that for all balls $B$
\begin{equation}\dstyle
\left( \frac{1}{|B|} \int_B \omega^r dx  \right)^{1/r}\leq \frac{C}{|B|} \int_B \omega dx.
\end{equation}
This would imply that any $A_p$ weight $\omega$ satisfies the doubling property: there exists a $C>0$ s.t. 
\begin{equation}
\int_{B(x_0, 2r)} \omega(x) dx \leq C \int_{B(x_0,r)} \omega (x) dx
\end{equation}
for all balls $B(x_0,r)$.

The notion of strong $A_\infty$ weight was first proposed by David and Semmes in \cite{DS}. Given a positive continuous weight $\omega$, we define 
\begin{equation}
\delta_\omega (x,y):= \left (  \int_{B_{x,y}}  \omega(z) dz \right )^{1/n},
\end{equation}
where $B_{x,y}$ is the ball with diameter $|x-y|$ that contains $x$ and $y$. 
On the other hand, we can define the geodesic distance with respect to the weight $\omega$ to be
\begin{equation}d_\omega (x,y):= \inf_\gamma   \int_{\gamma}   \omega^{\frac{1}{n}}(s) ds.\end{equation}
Here $\gamma\subset B_{x,y}$ is a curve connecting $x, y$ such that the tangent vector is always contact.
If $\omega$ is an $A_\infty$ weight, then it is easy to prove (see for example Proposition 3.12 in \cite{S2}) 
\begin{equation}
d_\omega(x,y) \leq C\delta_\omega (x,y)
\end{equation}
for all $x,y \in \mathbb H^n$. If in addition, $\omega$ also satisfies the reverse in equality
\begin{equation}
\delta_\omega(x,y)\leq C d_{\omega}(x,y)
\end{equation}
then we say $\omega$ is a strong $A_\infty$ weight.

The product of an $A_1$ weight and an $A_\infty$ weight is an $A_\infty$ weight. This can be proved using the same proof as in the Euclidean space.

\section{CR-manifold with negative $Q'$-curvature}\label{sect:negative}
In this section, we will prove Theorem \ref{main2}. It shows that for CR-manifolds with negative $Q'$-curvature, the integral of $Q'$-curvature controls the geometry in a very rigid way.

We first remark that since $Q'(y) e^{4u(y)}$ is integrable, $ \log \frac{|y|}{|x-y|} Q'(y) e^{nu(y)}$ is also integrable in $y$ for each fixed $x\in \mathbb{H}^1$. 

In this section, we consider the analytic property of $e^{4u(x)}$. For simplicity, we denote it by $\omega_2(x)$.
We define $ \beta:= \int_{\mathbb{H}^1}|Q'|(y)e^{4u(y)} dy<\infty$.
Recall that for a nonnegative continuous function $\omega(x)$,
$$d_{\omega}(x,y):= (\int_{B_{xy} } \omega(z)dz )^{\frac{1}{n}},
$$
$$\delta_{\omega}(x,y):=\inf_{\gamma} \displaystyle \int_{\gamma}\omega^{\frac{1}{n}}(\gamma(s))ds,
$$
where $B_{xy}$ is the ball with diameter $|x-y|$ that contains $x$ and $y$,
the infimum is taken over all contact curves (meaning that the tangent vector on each point of this curve is contact) $\gamma\subset B_{xy}$ connecting $x$ and $y$, and $ds$ is the arc length.

We want to prove $\omega_2(x):=e^{4u(x)}$ is a strong $A_\infty$ weight,
i.e. there exists a constant $C=C(\beta)$ such that
\begin{equation}\label{Ainfty2}
\frac{1}{C(\beta)} d_{\omega_2}(x,y)
\leq \delta_{\omega_2}(x,y)\leq C(\beta) d_{\omega_2}(x,y).
\end{equation}
Since the Webster scalar curvature is nonnegative at infiniity, by Proposition in \cite{WY}, $u$ is normal. Thus 
\begin{equation}
u(x)=\displaystyle \frac{-1}{c_1'}\int_{\mathbb{H}^1}\log \frac{|y|}{| x- y|} |Q'|(y)e^{4u( y)} dy.\\
\end{equation}

We first observe that without generality we can assume $|x-y|=2$. This is because we can dilate $u$ by a factor $\lambda>0$,
\begin{equation}
\begin{split}
u^\lambda(x):= u(\lambda x)=&
\displaystyle \frac{-1}{c_1'}\int_{\mathbb{H}^1}\log \frac{|y|}{|\lambda x- y|} |Q'|(y)e^{4u( y)} dy.\\
\end{split}
\end{equation}
By the change of variable, this is equal to
$$\displaystyle \frac{-1}{c_1'}\int_{\mathbb{H}^1}\log \frac{|y|}{|x-y|} |Q'|(\lambda y)e^{4u(\lambda y)}\lambda^4 dy.$$
Notice $|Q'|(\lambda y)e^{4u(\lambda y)}\lambda^4$ is still an integrable function on $\mathbb{H}^1$, with integral equal to $\beta$. Thus by choosing $\lambda =\frac{2}{|x-y|}$, the problem reduces to proving inequality (\ref{Ainfty2}) for $u^\lambda $ and $|x-y|=2$.


Let us denote the midpoint of $x$ and $y$ by $p_0$. And from now on, we adopt the notation
$\lambda B:=B (p_0,\lambda )$. Since $|x-y|=2$, we have $B_{xy}= B(p_0, 1)= B$.
We also define
\begin{equation}
u_1(x):=\displaystyle \frac{-1}{c_1'}\int_{10B}\log \frac{|y|}{|x-y|} |Q'|(y)e^{4u(y)} dy,
\end{equation}
and
\begin{equation}
u_2(x):=\displaystyle \frac{-1}{c_1'}\int_{\mathbb{H}^1\setminus 10 B} \log \frac{|y|}{|x-y|} |Q'|
(y)e^{4u(y)} dy.
\end{equation}

In the following lemma, we prove that when $z$ is close to $p_0$, the difference between $u_2(z)$ and $u_{2}(p_0)$ is controlled by $\beta$.
\begin{lemma}\label{claim1}
\begin{equation}\label{3.1}
|u_2(z)-u_{2}(p_0)|\leq \frac{ \beta}{4c_1'}
\end{equation}
for $z\in 2 B $.
\end{lemma}

\begin{proof}
\begin{equation}
\begin{split}
&|u_2(z)-u_{2}(p_0)|\\
=& \frac{1}{c_1'}\left|\int_{\mathbb{H}^1\setminus 10 B}-\log \frac{|y|}{|z-y|} |Q'|(y)e^{4u(y)}dy+  \int_{\mathbb{H}^1\setminus 10 B}\log \frac{|y|}{|p_0-y|} |Q'|(y)e^{4u(y)}dy\right |\\
=&\frac{1}{c_1'}\left|\int_{\mathbb{H}^1\setminus 10 B}\log \frac{|z-y|}{|p_0-y|} |Q'|(y)e^{4u(y)}dy\right|\\
\leq &\frac{|z-p_0|}{c_1'}\cdot \int_{\mathbb{H}^1\setminus 10 B}\frac{1}{|(1-t^*)(p_0-y)+ t^*(z-y)|}|Q'|(y)e^{4u(y)}dy, \\
\end{split}
\end{equation}
for some $t^*\in[0,1]$. Since $y\in \mathbb{H}^1\setminus 10 B$ and $z, p_0\in 2B$,
$$\frac{1}{|(1-t^*)(p_0-y)+ t^*(z-y)|}\leq \frac{1}{8},$$
$|u_2(z)-u_{2}(p_0)|$ is bounded by
\begin{equation}
\begin{split}
\frac{|z-p_0|}{8 c_1' }\cdot \int_{\mathbb{H}^1\setminus 10 B}|Q'|(y)e^{4u(y)}dy.
\end{split}
\end{equation}
Note that for $z\in 2B$, $|z-p_0|\leq 2 $. From this, (\ref{3.1}) follows.
\end{proof}

Now we adopt some techniques used in \cite{BHS}
for potentials to deal with the $\epsilon$-singular set $E_\epsilon$.
\begin{lemma}\label{Hauss1measure}(Cartan's lemma) For the Radon measure $|Q'|(y)e^{4u(y)}dy$, given $\epsilon>0$, there exists a set $E_\epsilon\subseteq \mathbb{H}^1$, such that
$$\mathcal{H}^1(E_\epsilon):= \dstyle \inf_{ E_\epsilon \subseteq \cup B_i }\{\dstyle \sum_{i}\mbox{diam } B_i \}< 10\epsilon $$
and for all $x\notin E_\epsilon$ and $r>0$,
$$\dstyle \int_{B(x,r)} |Q'|(y)e^{4u(y)}dy\leq \frac{ r \beta}{\epsilon }. $$
\end{lemma}
The proof of Lemma 1 follows from standard measure theory argument. Thus we omit it here.

\begin{proposition}\label{claim2} Given $\epsilon>0$,
$$  \mathcal{H}^1\left( \left\{x\in 10 B :\left |\frac{-1}{c_1'} \int_{10 B}\log \frac{1}{|x-y|} |Q'|(y)e^{4u}(y)dy  \right| >\frac{C_0\beta }{\epsilon}\right\}\right)< 10\epsilon.$$.
\end{proposition}
\begin{proof}
Fix $\epsilon>0$. By Lemma \ref{Hauss1measure}, there exists a set $E_\epsilon\subseteq \mathbb{H}^1$, s.t.
$\mathcal{H}^1(E_\epsilon)< 10\epsilon $ and
for $x\notin E_\epsilon$ and $r>0$
\begin{equation}\label{3.3}\dstyle \int_{B(x,r)}|Q'|(y) e^{4u(y)} dy \leq \frac{ r\beta}{\epsilon }.\end{equation}
If we can show for some $C_0$
\begin{equation}\label{3.2}\dstyle 10 B \setminus E_\epsilon\subseteq \left\{ x\in 10 B :\left| \frac{-1}{c_1'}\int_{10 B}\log \frac{1}{|x-y|}|Q'|(y)e^{4u(y)}dy\right| \leq \frac{C_0}{\epsilon} \beta\right\}, \end{equation}
then
$$\mathcal{H}^1\left( \left\{x\in 10 B :\left |\frac{-1}{c_1'} \int_{10 B}\log \frac{1}{|x-y|} |Q'|(y)e^{4u}(y)dy  \right| >\frac{C_0\beta }{\epsilon}\right\}\right)\leq  \mathcal{H}^1(E_\epsilon)<\\
10\epsilon.$$
To prove (\ref{3.2}), we notice for $x\in 10 B\setminus E_\epsilon$, $r= 2^{-j}\cdot 10$, (\ref{3.3}) implies
\begin{equation}
\begin{split}
&\left|\dstyle \frac{-1}{c_1'}\int_{10 B} \log \frac{1}{|x-y|} |Q'|(y)e^{4u(y)} dy\right|\\
\leq &\frac{1}{c_1'} \dstyle \sum_{j=-1}^\infty   \left|\int_{B(x, 2^{-j}\cdot 10)\setminus B(x, 2^{-(j+1)}\cdot 10)} \log \frac{1}{|x-y|} |Q'|(y)e^{4u(y)} dy\right|\\
\leq&\dstyle \frac{1}{c_1'}  \sum_{j=-1}^\infty \left(\max\{|\log 2^{-j}|,|\log 2^{-(j+1)}|\} + \log 10\right) \cdot \\
&\hspace{20mm}\int_{B(x, 2^{-j}\cdot 10)\setminus B(x, 2^{-(j+1)}\cdot 10)} |Q'|(y)e^{4u(y)}dy\\
\leq &\dstyle \frac{1}{c_1'}  \sum_{j=-1}^\infty \left(\max\{|\log 2^{-j}|,|\log 2^{-(j+1)}|\} + \log 10\right) \cdot  \frac{2^{-j}\cdot 10 \beta}{\epsilon }\\
\leq & \frac{C_0\beta}{\epsilon },\\
\end{split}
\end{equation}
where $$C_0=\frac{ 10 \sum_{j=-1}^\infty \left(\max\{|\log 2^{-j}|,|\log 2^{-(j+1)}|\}+ \log 10\right) \cdot 2^{-j}}{c_1'}<\infty. $$
This completes the proof of the proposition.
\end{proof}
We next estimate the integral of $e^{4u(z)} $ over $2B$.
\begin{proposition}\label{3.8}Let $\bar{c}:= \frac{-1}{c_1'}\int_{10B} \log |y||Q'|(y)e^{4u(y)}dy$. $\bar{c}<\infty$, since $|Q'|(y)e^{4u(y)}$ is continuous thus bounded near the origin. Then
\begin{equation}\label{3.4}
\int_{2B} e^{4u(z)}dz\leq C_1(\beta) e^{4u_2(p_0)} e^{4\bar{c}},
\end{equation}
for $C_1$
depends only on $\beta$.
\end{proposition}

\begin{proof}Recall
\begin{equation}
u_1(x):=\displaystyle \frac{-1}{c_1'}\int_{10B}\log \frac{|y|}{|x-y|} |Q'|(y)e^{4u(y)} dy,
\end{equation}
and
\begin{equation}
u_2(x):=\displaystyle \frac{-1}{c_1'}\int_{\mathbb{H}^1\setminus 10 B} \log \frac{|y|}{|x-y|} |Q'|(y)e^{4u(y)} dy.
\end{equation}
By Lemma \ref{claim1},
\begin{equation}\label{3.12}
\begin{split}
\int_{2B} e^{4u(z)}dz=&\int_{2B} e^{4u_1(z)}e^{4u_2(z)}dz\\
\leq&  e^{\frac{\beta}{c_1'}}e^{4u_2(p_0)}\int_{2B} e^{4u_1(z)}dz.\\
\end{split}
\end{equation}

To estimate $u_1$, by definition
$\beta_{10}:=\int_{10B} |Q'|(y)e^{4u(y)} dy\leq \beta<\infty $.
If $\beta_{10}= 0$, then $u_1(z)=0$ and $\bar{c}:= \frac{-1}{c_1'}\int_{10B} \log |y||Q'|(y)e^{4u(y)}dy=0$. So (\ref{3.4}) follows immediately. If $\beta_{10}\neq 0$,
$\frac{|Q'|(y)e^{4u(y)}}{\beta_{10}} dy$ is a nonnegative probability measure on $10 B$. Hence by Jensen's inequality
\begin{equation}
\begin{split}
\int_{2B} e^{4u_1(z)}dz= &e^{4\bar{c}}\cdot \int_{2B} e^{\frac{4}{c_1'} \int_{10B}(\log |z-y|) |Q'|(y)e^{4u(y) } dy } dz\\
\leq &e^{4\bar{c}}\cdot \int_{2B}  \int_{10B}|z-y|^{\frac{4\beta_{10}}{c_1'}} \frac{|Q'|(y)e^{4u(y)} }{\beta_{10}}dy  dz.\\
\end{split}
\end{equation}
Since $z\in 2B$ and $y\in 10B$,
\begin{equation}
 \int_{2B}|z-y|^{\frac{4\beta_{10}}{c_1'}}dz\leq C.
\end{equation}
From this, we get
\begin{equation}
\begin{split}
\int_{2B} e^{4u_1(z)}dz\leq &  C e^{4\bar{c}}  \int_{10B} \frac{|Q'|(y)e^{4u(y)}}{\beta_{10}}dy=C e^{4\bar{c}} .\\
\end{split}
\end{equation}
Plugging it to (\ref{3.12}), we finish the proof of the proposition.
\end{proof}
Now we are ready to prove Theorem \ref{main2}.\\
\noindent{\it{Proof of Theorem \ref{main2}.}}
Let us assume $\omega_2:=e^{4u}$ is an $A_p$ weight for some large $p$, with bounds depending only on
$\beta$. The proof of this fact follows that of Proposision 5.1 in \cite{YW15}. So we omit it here. 
By the reverse H\"{o}lder's inequality for $A_p$ weights, it is easy to prove (see for example Proposition 3.12 in \cite{S2}),
$$\delta_{\omega_2}(x,y)\leq C_2(\beta) d_{\omega_2}(x,y).$$
Hence we only need to prove the other side of the inequality:
\begin{equation}\label{3.9}\delta_{\omega_2}(x,y)\geq C_3(\beta) d_{\omega_2}(x,y),
\end{equation}
for some constant $C_3(\beta)$.
By Proposition \ref{claim2}, for a given $\epsilon>0$, there exists a Borel set $E_\epsilon\subseteq
\mathbb{H}^1$, such that
\begin{equation}\label{3.13}\mathcal{H}^{1}(E_\epsilon)\leq 10 \epsilon,\end{equation} and for
$z\in 10B \setminus E_\epsilon$, according to (\ref{3.2})
\begin{equation}
\label{3.7}
|\hat{u}_1(z)|\leq \frac{C_0}{\epsilon}\beta. \end{equation}
Here $$\hat{u}_1(z):=\dstyle \frac{-1}{c_1'}\int_{10 B}\log \frac{1}{|x-y|} |Q'|(y)e^{4u(y)}dy.  $$
With this, we claim the following estimate.\\
{{\bf Claim:} Suppose $\mathcal{H}^{1}(E_\epsilon)< 10\epsilon$ with $\epsilon\leq \frac{1}{20}$. Then the  length of $\gamma\setminus E_\epsilon$ with respect to the metric of Heisenberg group $\mathbb{H}^1$ satisfies
\begin{equation}\label{arclength-est}
\dstyle \mbox{length } (\gamma\setminus E_\epsilon)> \frac{3}{2},
\end{equation}
where $\gamma\subset B_{xy}$ is a curve connecting $x$ and $y$.
}

{\it {Proof of Claim.}}
Let $P$ be the projection map from points in $B_{xy}$ to the contact line segment $I_{xy}$ between $x$ and $y$. Since the Jacobian of the projection map is less or equal to 1,
\begin{equation}\label{3.5}
\mbox{length } (\gamma\setminus E_\epsilon)\geq
\mbox{length } (P(\gamma\setminus E_\epsilon))= m(P(\gamma\setminus E_\epsilon)),
\end{equation}
where $m$ is the arc length measure on the line segment $I_{xy}$.
Notice $P(\gamma)=I_{xy}$, and
$P(\gamma)\setminus P(E_\epsilon)$ is a subset of $P(\gamma\setminus E_\epsilon)$. Therefore
\begin{equation}\label{3.6}
m(P(\gamma\setminus E_\epsilon))\geq m(P(\gamma))- m(P(E_\epsilon))=2-m(P(E_\epsilon)).\end{equation}
Now by assumption, $\mathcal{H}^{1}(E_\epsilon)< 10\epsilon$,
 so $\mathcal{H}^{1}(\gamma\cap E_\epsilon)< 10\epsilon$. Hence
there is a covering $\cup_i B_i$ of $\gamma \cap E_\epsilon$, so that
$$ \dstyle \sum_{i}\mbox{diam } B_i< 10\epsilon. $$
This implies that $\cup_{i} P(B_{i})$ is a covering of the set $P(\gamma \cap E_\epsilon)$
and
$$\dstyle \sum_{i}\mbox{diam } P(B_i)= \dstyle \sum_{i}\mbox{diam } B_i\leq 10\epsilon. $$
Thus $m(P(E_\epsilon))= \mathcal{H}^1(P(E_\epsilon))< 10\epsilon< \frac{1}{2}$, by choosing $\epsilon \leq \frac{1}{20}$. Plug it to (\ref{3.6}), and then to (\ref{3.5}). This completes the proof of the claim.

We now continue the proof of Theorem \ref{main2}.
Since $\gamma \subset B$, then by Lemma \ref{claim1},
\begin{equation}
\begin{split}
\dstyle \int_{\gamma} e^{u_-(\gamma(s))} ds=\int_{\gamma} e^{(u_1+u_2)(\gamma(s))} ds \geq &\dstyle e^{\frac{-\beta}{4c_1'}}e^{u_{2}(p_0)}  e^{\bar{c}}  \int_{\gamma}  e^{\hat{u}_1(\gamma(s))}ds.\\
\end{split}
\end{equation}
Here $\bar{c}$ is the constant defined in Proposition \ref{3.8}.
Let $\epsilon =\frac{1}{20}$. By (\ref{3.7}),
$$|\hat{u}_1(z)|\leq 20C_0 \beta $$
for $z\in 10B\setminus E_\epsilon$.
Thus
\begin{equation}
\dstyle \int_{\gamma}  e^{\hat{u}_1(\gamma(s))}ds\geq e^{- 20 C_0 \beta} \mbox{length } (\gamma\setminus E_\epsilon) .\\
\end{equation}
By (\ref{arclength-est}), it is bigger than
$$\frac{3}{2}e^{- 20 C_0 \beta}. $$
Therefore
\begin{equation}\label{3.10}
\begin{split}
\dstyle \int_{\gamma} e^{u_-(\gamma(s))} ds\geq
\frac{3}{2} e^{\frac{-\beta}{4c_1'}} e^{- 20  C_0 \beta}e^{u_{2}(p_0)}  e^{\bar{c}}=C_4(\beta)e^{u_{2}(p_0)}  e^{\bar{c}}
\end{split}
\end{equation}
for $C_4(\beta)=\frac{3}{2} e^{\frac{-\beta}{4c_1'}} e^{- 20  C_0 \beta}$.
By inequality (\ref{3.10}) and Proposition \ref{3.8}, we conclude for any curve $\gamma\subset B_{xy}$ connecting $x$ and $y$, there is a $C_3=C_3(\beta)$ such that
\begin{equation}
\begin{split}
\dstyle \int_{\gamma} e^{u_-(\gamma(s))} ds\geq
C_{3}(\beta) (\int_{B_{xy}} e^{4u_-(z)}dz)^{\frac{1}{4}}.
\end{split}
\end{equation}
This implies inequality (\ref{3.9}) and thus completes the proof of Theorem \ref{main2}.

\section{ $Q'$-curvature without a sign} 
In this section, we consider CR-manifold on which the $Q'$-curvature does not have a sign any more.
Suppose $(\mathbb{H}^1, e^{2u}\theta)$ satisfies that 
\begin{equation}
\alpha:=\displaystyle \int_{\mathbb{H}^1} Q'^+ e^{4u}\theta \wedge d\theta<c_1', 
\end{equation}
 \begin{equation}
\beta:=\displaystyle \int_{\mathbb{H}^1} Q'^- e^{4u}\theta \wedge d\theta<\infty. 
\end{equation}
Suppose also that the Webster scalar curvature is nonnegative at infinity.

By Theorem \ref{main2}, $e^{4u^-}$ is a strong $A_\infty$ weight.
By Theorem 1.4 in \cite{WY}, $e^{4u^+}$ is an $A_1$ weight. 

\begin{proposition}\label{prop:3.1}
Assume $\omega_1$ is an $A_1$ weight, $\omega_2$ is a strong $A_\infty$ weight. If $\omega_1^r \omega_2$ for some $r\in \mathbb{R}$ is an $A_\infty$ weight, then $\omega_1^r \omega_2$ a strong $A_\infty$ weight.
\end{proposition}

\begin{rem}The proposition for the Euclidean space has been proved in \cite{S2}.
We prove here the proposition for Heisenberg groups. 
\end{rem}

\begin{proof}
Let $\delta_2(\cdot, \cdot)$ and $\delta_{12}(\cdot, \cdot)$ be the quasidistance associated to $\omega_2$ and  $\omega_1^r \omega_2$ respectively. Let $x_1,...,x_k \in \mathbb{H}^1$ such that $x_j\in B(x_1, 2|x_k-x_1|)$ for all $j$. Notice that it suffices to prove 
\begin{equation}
\displaystyle \delta_{12}(x_1,x_k)\leq C \sum_{j=1}^{k-1}\delta_{12} (x_j,x_{j+1}).
\end{equation}
Let $B=B_{x_1,x_k}$, and $B_j=B_{x_j,x_{j+1}}$. Since $x_j\in B(x_1, 2|x_k-x_1|)$ for all $j$, $B_j\subset 100B$ for all $j$. By definition $\delta_{12}$,
\begin{equation}
\delta_{12}(x_j,x_{j+1})
\end{equation}

\end{proof}

By Proposition \ref{prop:3.1}, in order to prove Theorem \ref{main3}, we only need to show that $e^{4u}$ is an $A_\infty$ weight. In other words, we need to show $e^{4u}$ is an $A_p$ weight for some $p$.
 
\begin{proposition}
Suppose $(\mathbb{H}^1, e^{2u}\theta)$  satisfies the same assumptions as in Theorem \ref{main3}. Then
$e^{4u}$ is an $A_p$ weight for some $p$. The $A_p$ bound depends only on the integral of $Q'$ curvature.
\end{proposition}

\begin{proof}
\begin{equation}u(x)=
\displaystyle \frac{1}{c_1'}\int_{\mathbb H^1} \log \frac{|y|}{|x-y|} Q'(y)e^{4u(y)} dy
\end{equation}
with assumptions (\ref{assumption1}) and (\ref{assumption2}),
By Theorem 1.4 in \cite{WY}, $e^{4u_+}$ is an $A_1$ weight, so there is a uniform constant $C=C(\alpha)$, so that for all $x_0\in \mathbb H^1$ and $r>0$
\begin{equation}
\dstyle \frac{1}{|B(x_0,r)|} \int_{B(x_0,r)} e^{4u_+(y)} dy \leq C(\alpha) e^{4u_+(x_0)}.
\end{equation}
So for all $x\in B(x_0,r)$
\begin{equation}
\begin{split}
 \dstyle \frac{1}{|B(x_0,r)|} \int_{B(x_0,r)} e^{4u_+(y)} dy \leq& \dstyle \frac{1}{|B(x_0,r)|} \int_{B(x,2r)} e^{4u_+(y)} dy\\
&=  \frac{2^4}{|B(x,2r)|} \int_{B(x,2r)} e^{4u_+(y)} dy\\
&\leq  C(\alpha) e^{4u_+(x)}.\\
\end{split}
\end{equation}
Namely, for all ball $B$ in $\mathbb H^1$ and $x\in B$,
\begin{equation}\label{5.3}
 \dstyle \frac{1}{|B|} \int_{B} e^{4u_+(y)} dy \leq C(\alpha) e^{4u_+(x)}.
\end{equation}
We observe that $e^{ -4 \epsilon  u_-(x)}  $ is also an $A_1$ weight for $\epsilon= \epsilon (\beta)<<1$.
In fact,
\begin{equation}e^{ -4\epsilon  u_-(x)}   = e^{ \frac{4}{c_1'}\int_{\mathbb H^1} \log\frac{|y|}{|x-y|} \epsilon Q^{ -}(y) e^{4u(y)} dy   }. \end{equation}
$ Q^{ -}(y) e^{4u(y)}\geq 0 $ and $\int_{\mathbb H^1} \epsilon Q^{ -}(y) e^{4u(y)} dy< c_1'$ if $\epsilon$ is small enough.
Thus by Theorem 1.4 in \cite{WY}, $e^{ -4 \epsilon  u_-(x)}  $ is an $A_1$ weight. As (\ref{5.3}), we have
\begin{equation}\label{5.4}
\dstyle \frac{1}{|B|} \int_{B} e^{ -4 \epsilon  u_-(y)}     dy \leq C(\beta) e^{ -4 \epsilon  u_-(x)}
\end{equation}
for all ball $B$ in $\mathbb H^1$ and all $x\in B$.
Choose $1<p<\infty$ such that $\epsilon= p'/p$ with $\frac{1}{p}+ \frac{1}{p'}= 1$.
Using $ e^{4u}= e^{4u_+}\cdot e^{4u_-} $, we get
\begin{equation}\label{5.16}
\begin{split}
&\left(\int_{B}  e^{4u(x)} dx\right)^{\frac{1}{p}}\left(\int_{B} (e^{4u(x)})^{-\frac{p'}{p}} dx\right)^{\frac{1}{p'}}\\
= &\left(\int_{B}  e^{4u_+}\cdot (e^{-4\epsilon  u_-} )^{-\frac{1}{\epsilon}}   dx\right)^{\frac{1}{p}} \left(\int_{B} (e^{ 4u_+})^{-\frac{p'}{p}} \cdot e^{-4\epsilon  u_-}   dx\right)^{\frac{1}{p'}}.\\
\end{split}
\end{equation}
By (\ref{5.4}), if $p$ is large enough and thus $\epsilon$ is small enough, then
$$(e^{-4\epsilon  u_-} )^{-\frac{1}{\epsilon}} \leq  \left(\frac{1}{C(\beta) |B|}\int_{B}e^{-4\epsilon  u_-}dx\right)^{-\frac{1}{\epsilon }}.$$
So
\begin{equation}\label{5.14}
\begin{split} \left(\int_{B}  e^{4u_+}\cdot (e^{-4\epsilon  u_-} )^{-\frac{1}{\epsilon}}   dx\right)^{\frac{1}{p}}
\leq& \left(\int_{B}  e^{4u_+}    dx\right)^{\frac{1}{p}}   \left(\frac{1}{C(\beta) |B|}\int_{B}e^{-4\epsilon  u_-}dx\right)^{-\frac{1}{\epsilon p}} \\
 =& \left(\int_{B}  e^{4u_+}    dx\right)^{\frac{1}{p}}   \left(\frac{1}{C(\beta)  |B|}\int_{B}e^{-4\epsilon  u_-}dx\right)^{-\frac{1}{p'}}.\\
\end{split}
\end{equation}
Similarly, by (\ref{5.3})
 $$(e^{ 4u_+})^{-\frac{p'}{p}}\leq \left(\frac{1}{C(\alpha)  |B|}\int_{B} e^{ 4u_+} dx \right)^{-\frac{p'}{p}}.$$
So \begin{equation}\label{5.15}
\left(\int_{B} (e^{ 4u_+})^{-\frac{p'}{p}} \cdot e^{- 4 \epsilon  u_-}  dx\right)^{\frac{1}{p'}}
\leq
\left(\frac{1}{ C(\alpha)   |B|}\int_{B} e^{ 4u_+} dx \right)^{-\frac{1}{p}}
\left(\int_{B}  e^{-4\epsilon  u_-} dx\right)^{\frac{1}{p'}}.
 \end{equation}
Applying (\ref{5.14}) to (\ref{5.15}) in (\ref{5.16}), we have
\begin{equation}
\left(\int_{B}  e^{4u(x)} dx\right)^{\frac{1}{p}}\left(\int_{B} (e^{4u(x)})^{-\frac{p'}{p}} dx\right)^{\frac{1}{p'}}\\
\leq (\frac{1}{C |B|})^{-\frac{1}{p}-\frac{1}{p'}}=C|B|
\end{equation}
for $p>>1$.
This shows that $ e^{4u(x)}$ is an $A_p$ weight for $p>>1$. The bound $C$ depends only on $\alpha$ and $\beta$.
\end{proof}

\section{Proof of Theorem \ref{main4}}
\begin{theorem}\cite[Theorem 2]{FLW} Let $\{X_j\}$ be a family of vector fields that satisfies H\"ormander's condition.
 Let $K$ be a compact subset of $\Omega$. Then there exists $r_0$ depending on $K$, $\Omega$ and $\{X_j\}$ such that if $B=B(x,r)$ is a ball with $x\in K$ and $0<r< r_0$, and if $1\leq p<q<\infty$ and $\omega_1$, $\omega_2$ are weights satisfying the balance condition \eqref{FLW1.5} for $B$, with 
 $\omega_1\in A_p (\Omega, \rho, dx)$ and $\omega_2$ doubling, then
 \begin{equation}
 (\frac{1}{\omega_2(B)} \int_B |f(x)- f_B |^q \omega_2(x)dx	)^{1/q} \leq cr (\frac{1}{\omega_1(B)} \int_B |Xf(x)|^p \omega_1(x)dx)^{1/p}
 \end{equation}
for any $f\in Lip (\bar{B})$, with $f_B= \omega_2 (B)^{-1} \int_B f(x) \omega_2(x) dx$. The constant $c$ depends only on $K, \Omega, \{X_j\}$ and the constants in the conditions imposed on $\omega_1$, and $\omega_2$. 

\end{theorem}

The balance condition is stated as follows:
for two weight functions $\omega_1$, $\omega_2$ on $\Omega$ and $1\leq p< q<\infty$, a ball $B$ with center in $K$ and $r(B)<r_0$:
\begin{equation}\label{FLW1.5}
\frac{r(I)}{r(J)} (\frac{\omega_2 (I)}{\omega_2 (J)})^{1/q} \leq c  (\frac{\omega_1 (I)}{\omega_1 (J)})^{1/p} 
\end{equation}
for all metric balls $I, J$ with $I\subset J \subset B$.

\begin{proof} of Theorem \ref{main4}. It is obvious that 
$X_1:= \frac{\partial }{\partial x} +2y\frac{\partial }{\partial t}$, $X_2:= \frac{\partial }{\partial y} -2 x\frac{\partial }{\partial t}$ on the Heisenberg group $\mathbb H^1$ satisfy the H\"omander's condition. Let us take $\omega_1(x)= e^{(n-p)u(x)}$, $\omega_2(x)= e^{nu(x)}$, $q= \frac{np}{n-p}$.

We only need to check condition \eqref{FLW1.5}. Namely, we need to show
\begin{equation}\label{ineqn:5.3}
(\frac{r(I)}{r(J)})^\frac{np}{n-p} \frac{\int_I \omega_2 dx }{\int_J \omega_2 dx}  \leq c  (\frac{\int_I \omega^{\frac{n-p}{n}}_2 dx }{\int_J  \omega^{\frac{n-p}{n}}_2dx}) 
\end{equation}

This is true because $0\leq {\frac{n-p}{n}}<1$ and $\omega_2= e^{nu}$ is a strong $A_\infty$ weight, thus it is an $A_\infty$ weight. 
In fact, for any $A_\infty$ weight $w$,  $0\leq s<1$, by the result of Str\"omberg-Wheeden \cite{SW}
\begin{equation}
(\frac{1}{|B|} \int_B w(x)^s dx)^{\frac{1}{s}} \leq  C \frac{1}{|B| } \int_B w(x) dx.
\end{equation}
On the other hand, by H\"older's inequality,
\begin{equation}
 \frac{1}{|B| } \int_B w(x) dx\leq (\frac{1}{|B|} \int_B w(x)^{s} dx)^{\frac{1}{s}} 
\end{equation}
Therefore by taking $s=\frac{n-p}{n}$, \eqref{ineqn:5.3} holds.

\end{proof}



\end{document}